\documentclass[graybox]{svmult}

\usepackage{type1cm}        
\usepackage{makeidx}         
\usepackage{graphicx}        
\usepackage{multicol}        
\usepackage[bottom]{footmisc}

\usepackage{newtxtext}       %
\usepackage[varvw]{newtxmath}       

\makeindex             
\usepackage{url}

\DeclareMathOperator{\prox}{prox}

\DeclareMathOperator*{\argmin}{arg\,min}
\newcommand{\R}{\mathbb{R}}
\newcommand{\N}{\mathbb{N}}
\newcommand{\dd}{\text{d}}

\begin{document}

\title*{On a fixed-point continuation method for a convex optimization problem}
\author{Jean-Baptiste Fest, Tommi Heikkilä, Ignace Loris, Ségolène Martin, Luca Ratti, Simone Rebegoldi and Gesa Sarnighausen}

\authorrunning{Fest, Heikkilä, Loris, Martin, Ratti, Rebegoldi and Sarnighausen}

\institute{Jean-Baptiste Fest \at Centre de Vision Numérique, Inria, CentraleSupélec, Université Paris-Saclay, 3 Rue Joliot Curie, 91190 Gif-sur-Yvette \email{jean-baptiste.fest@centralesupelec.fr}
\and Tommi Heikkilä \at Unversity of Helsinki, Department of Mathematics and Statistics.
Pietari Kalmin Katu 5, 00560 Helsinki, Finland, \email{tommi.heikkila@helsinki.fi}
\and Ignace Loris \at Université libre de Bruxelles, Brussels, Belgium \email{Ignace.Loris@ulb.be}
\and Ségolène Martin \at Centre de Vision Numérique, Inria, CentraleSupélec, Université Paris-Saclay, 3 Rue Joliot Curie, 91190 Gif-sur-Yvette \email{segolene.martin@centralesupelec.fr}
\and Luca Ratti \at Machine Learning Genoa Center (MaLGa), Università degli studi di Genova, Via Dodecaneso 35, 16146 Genova, Italy \email{luca.ratti@unige.it}
\and Simone Rebegoldi \at Università degli studi di Firenze, Viale G.B. Morgagni 40, 50134 Firenze, Italy; INDAM-GNCS Research group, Roma, Italy \email{simone.rebegoldi@unifi.it}
\and Gesa Sarnighausen \at Georg-August-Universität Göttingen, Göttingen, Germany \email{gesa.sarnighausen@stud.uni-goettingen.de}}
%
%
\maketitle

\abstract{
We consider a variation of the classical proximal-gradient algorithm for the iterative minimization of a cost function consisting of a sum of two terms, one smooth and the other prox-simple, and whose relative weight is determined by a penalty parameter. This so-called fixed-point continuation method allows one to approximate the problem's trade-off curve, i.e. to compute the minimizers of the cost function for a whole range of values of the penalty parameter at once.  The algorithm is shown to converge, and a rate of convergence of the cost function is also derived. Furthermore, it is shown that this method is related to iterative algorithms constructed on the basis of the $\epsilon$-subdifferential of the prox-simple term. Some numerical examples are provided.    
%
}

\section{Introduction}\label{sec:intro}

In this paper, we address the numerical and iterative solution of the following composite convex optimization problem.
\begin{problem}\label{problem}
Solve
\begin{equation}\label{eq:problem}
    \min_{u \in \R^d} F_\lambda(u) \equiv f(u)+\lambda g(u),
\end{equation}
where
\begin{itemize}
    \item $f:\R^d\rightarrow \R$ is convex and continuously differentiable;
    \item $\nabla f: \R^d \rightarrow \R^d$ is $L-$Lipschitz continuous;
    \item $g:\R^d\rightarrow \R\cup\{+\infty\}$ is convex, proper and lower semicontinuous;
    \item $\lambda >0$;
    \item $F_\lambda $ admits at least one minimum point $\hat u(\lambda)\in\R^d$.
\end{itemize}
\end{problem}
Furthermore, we assume that the gradient of $f$ is available for use in an iterative algorithm. Finally we also assume that the function $g$ is prox-simple, meaning that the proximal mapping $\prox_{\alpha g}$ \cite{Moreau1965,ComWa2005,Bauschke2011} can also be computed at each point of $\R^d$ and for each value of $\alpha>0$:
\begin{equation}
\prox_{\alpha g}(a)=\argmin_{u\in\R^d}\frac{1}{2}\|u-a\|_2^2+\alpha g(u).
\end{equation}

Under these conditions, the so-called proximal-gradient algorithm (and its generalizations and improvements) \cite{Beck-Teboulle-2009b,Bonettini-Loris-Porta-Prato-Rebegoldi-2017,Chen-et-al-2018,ComWa2005,Villa2013} can be applied to iteratively solve (\ref{eq:problem}). In its basic form, this algorithm reads as
\begin{equation}\label{eq:proxgrad}
\begin{cases}
    u_0 \in\R^d\\
    u_{n+1} = \prox_{\alpha \lambda g}(u_n-\alpha \nabla f(u_n)), \quad n=0,1,\ldots
\end{cases}
\end{equation}
and convergence of the sequence $(u_{n})_{n\in\N}$ to a minimizer of problem (\ref{eq:problem}) is guaranteed for any starting point $u_0$ when the step-size obeys $0<\alpha<2/L$ \cite{ComWa2005}.

Although the iterative method \eqref{eq:proxgrad} addresses 
the problem of the numerical computation of the minimizer $\hat u(\lambda)$ for a given value of $\lambda$, such a method  needs to be 
repeatedly applied if problem \eqref{eq:problem} is to be solved for several values of the penalty parameter $\lambda$. 
This is often the case when the cost function  (\ref{eq:problem}) appears in the modeling
of an inverse problem \cite{Bertero2021,Engl2000,Kirsch2011}, where the function $f$ represents a data misfit term and the function $g$ represents a penalty term that counter-balances 
the ill-posedness of the inverse problem. In this context the value of the penalty parameter $\lambda$ is not necessarily known in advance. Hence the question of a more efficient calculation of a whole family of minimizers $\{\hat u(\lambda) \ : \ \lambda_{\text{min}}\leq\lambda\leq \lambda_{\text{max}}\}$ of problem (\ref{eq:problem}) is of practical importance.

In this work, we are interested in studying the convergence properties of the sequence $(u_n)_{n\in\N}$ defined by the iterative procedure
\begin{equation}\label{eq:method}
\begin{cases}
    u_0 \in\R^d\\
    u_{n+1} = \prox_{\alpha \lambda_n g}(u_n-\alpha \nabla f(u_n)), \quad n=0,1,\ldots
\end{cases}
\end{equation}
where $\alpha>0$ is again a steplength parameter, and $\{\lambda_n\}_{n\in\N}\subseteq \R_{>0}$ is now a sequence of regularization parameters. 

Methods that employ the strategy $\lambda_n\rightarrow \lambda$ are usually called {\it fixed-point continuation algorithms}. The continuation strategy has been used in several applications in image processing, ranging from $\ell_1-$regularized minimization \cite{Hale2008}, low rank minimization \cite{Goldfarb2011}, plug-and-play algorithms \cite{Chan2017}, and tomography \cite{Bubba2020,Purisha2017}. In a more general context, approximate splitting algorithms as in \cite[Corollary 6.7]{Combettes2004} are closely related.Algorithm \eqref{eq:method} includes some of these fixed-point continuation algorithms as special cases, such as the ones in \cite{Hale2008,Goldfarb2011}. However, to the best of our knowledge, a convergence analysis of method \eqref{eq:method} (or its instances) is still missing. In this paper, we prove the convergence of the iterates of \eqref{eq:method} to a minimum point of problem \eqref{eq:problem}, as well as a rate of convergence on the function values. Furthermore, we shed light on the relation of \eqref{eq:method} with a certain class of inexact forward-backward algorithms.

The advantage of the modified (varying $\lambda_n$) proximal gradient algorithm (\ref{eq:method}) lies not in an accelerated convergence rate  as compared to the classical version (fixed $\lambda$), but in the observation that an adequate choice of the sequence $(\lambda_n)_{n\in\N}$ may enforce a useful path of the iterates in the penalty-misfit plane (the $g$-$f$ plane). Indeed, by starting the iteration with a minimizer of $F_{\lambda_\text{max}}$ and a large value of $\lambda_0=\lambda_\text{max}$, and by slowly decreasing $\lambda_n$ at every step, one can expect that each iterate $u_n$ is a good approximation of each minimizer $\hat u(\lambda_n)$ of the cost functions $F_{\lambda_n}$ up to $\lambda=\lambda_\text{min}=\lim_{n\to\infty}\lambda_n$. In this way, an approximation of the trade-off curve (also known as the L-curve \cite{Hansen2001}; see also \cite{Berg2008,Berg2011} and sections~\ref{sec:tradeoff} and \ref{sec:experiments}) can be made at the cost of computing just a single minimizer (for a single value of $\lambda$).

The paper is structured as follows. In section \ref{sec:tradeoff}, we study the properties of the trade-off curve, by which method \eqref{eq:method} is inspired. Section \ref{sec:convergence} includes the convergence analysis of the method and its connection to inexact forward--backward algorithms. In section \ref{sec:experiments}, we investigate the numerical approximation of the trade-off curve of a regularized least squares optimization problem. In section \ref{sec:conclusions}, we draw some conclusions related to our work.

\section{Trade-off curve}\label{sec:tradeoff}

Algorithm (\ref{eq:method}) draws its inspiration from an analysis of the so-called trade-off curve associated to problem (\ref{eq:problem}), and to the closely related constrained problem 
\begin{equation}\label{eq:constrainedproblem}
\min_{u\in\R^d} f(u) \quad\text{such that}\quad g(u)\leq \tau.
\end{equation}
We assume that a minimizer $\tilde u(\tau)$ of \eqref{eq:constrainedproblem} exists when the feasible set is non-empty. 
Let us introduce the value function 
\begin{equation}\label{eq:valuefunction}
\varphi(\tau)=\min_{u\in\R^d}\left\{ f(u) \quad\text{such that}\quad g(u)\leq \tau\right\}
\end{equation}
of this constrained problem and define the \emph{trade-off curve} (also known as the \emph{Pareto-curve}) as the graph of the value function (see e.g. Figure~\ref{fig1}, left panel). 

\begin{figure}
\centering\includegraphics{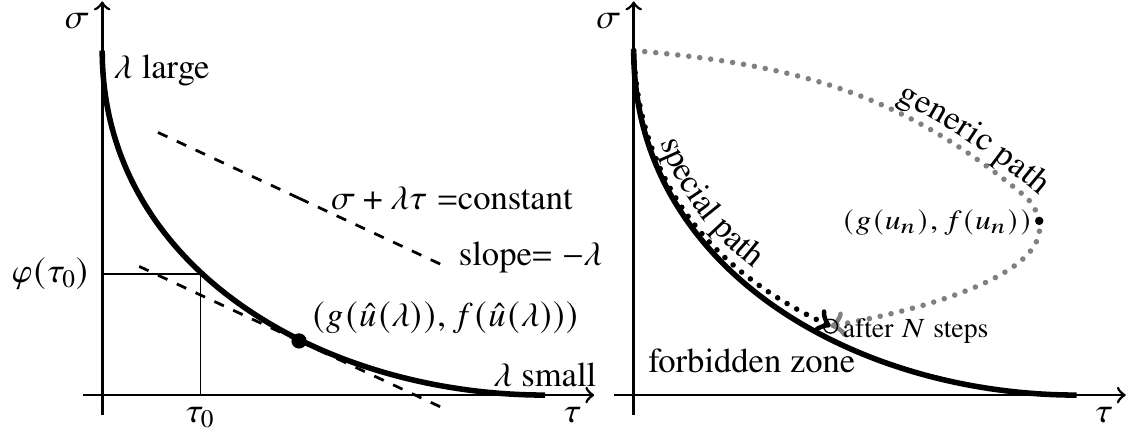}
\caption{Left: Graphical representation of the so-called trade-off curve and its relation to the penalty parameter $\lambda$. Right: Path (in the $g-f$-plane) of two different iterative optimization algorithms applied to the same instance of problem (\ref{eq:problem}) and starting from the same initial point. The black dotted path is special with respect to a generic path, as the former path approximately samples the trade-off curve (i.e. intermediate iterates have some interest) and the latter produces uninteresting intermediate iterates.}\label{fig1}
\end{figure}

\begin{property} If $f,g:\R^d\to\R\cup\{+\infty\}$ are convex, then the following statements hold true.
\begin{enumerate}
\item The value function $\varphi:\R\to\R\cup\{+\infty\}$ is non-increasing and convex.
\item The area below the curve  $(\tau,\varphi(\tau))_{\tau\in\R}$ cannot be reached by a point of the form $(g(u),f(u))$ with $u\in\R^d$.

\item If $\lambda\geq 0$, $\hat u(\lambda)$ is a solution of optimization problem (\ref{eq:problem}), and one sets $\tau=g(\hat u(\lambda))$, then $\hat u(\lambda)$ solves the constrained problem (\ref{eq:constrainedproblem}). Moreover, if $f$ and $g$ are differentiable, then $\varphi'(\tau)=-\lambda$, i.e., the slope of the trade-off curve equals $-\lambda$ at the point $(g(\hat u(\lambda)),f(\hat u(\lambda)))$.
\end{enumerate}
\end{property}
\begin{proof}
1) If $\tau_1\leq \tau_2$ one has $\{u$ s.t. $g(u)\leq\tau_1\}\subset \{u$ s.t. $g(u)\leq\tau_2\}$ and hence $\min\left\{ f(u)\ \text{s.t.}\ g(u)\leq \tau_1\right\}\geq \min\left\{ f(u)\ \text{s.t.}\ g(u)\leq \tau_2\right\}$.

Convexity is a well-known property of the value function \cite[p.~50]{Ekeland1999}. By definition of $\varphi$ one has:
\begin{displaymath}
\begin{array}{l}
\forall a_1 \text{ s. t. } \varphi(\tau_1)<a_1\quad \exists u_1\in\R^d \text{ s. t. } g(u_1)\leq \tau_1\text{ and } \varphi(\tau_1)\leq f(u_1)<a_1\\[3mm]
\forall a_2 \text{ s. t. } \varphi(\tau_2)<a_2\quad \exists u_2\in\R^d \text{ s. t. } g(u_2)\leq \tau_2\text{ and } \varphi(\tau_2)\leq f(u_2)<a_2
\end{array}
\end{displaymath}
which implies for $\mu\in[0,1]$:
\begin{displaymath}
\begin{array}{lcl}
 \displaystyle \varphi(\mu \tau_1+(1-\mu)\tau_2) & \stackrel{(\ref{eq:valuefunction})}{=} & \displaystyle 
\inf\left\{ f(u) \quad\text{with}\quad g(u)\leq \mu \tau_1+(1-\mu)\tau_2\right\} \\[3mm]
&\stackrel{g\text{ convex}}{\leq}& f(\mu u_1+(1-\mu)u_2)\\[3mm]
&\stackrel{f\text{ convex}}{\leq}& \mu  f(u_1)+(1-\mu)f(u_2)\\[3mm]
&<& \mu  a_1+(1-\mu)a_2.
\end{array}
\end{displaymath}
Since this holds for $a_1$ and $a_2$ arbitrarily close to $\varphi(\tau_1)$ and $\varphi(\tau_2)$ respectively, it follows that
\begin{displaymath}
\varphi(\mu \tau_1+(1-\mu)\tau_2)\leq \mu  \varphi(\tau_1)+(1-\mu)\varphi(\tau_2).
\end{displaymath}
Another proof is given in \cite[Theorem~2.1]{Berg2008}.

2) If there exists an element $u\in\R^d$ such that $g(u)=\tau$ and $f(u)<\varphi(\tau)$ then this is a contradiction with the definition (\ref{eq:valuefunction}) of the value function $\varphi$.

3) If $\hat{u}(\lambda)$ is a solution of (\ref{eq:problem}), then we have
\begin{displaymath}
0 \in \partial f(\hat u(\lambda)) + \lambda \partial g(\hat u(\lambda));
\end{displaymath}
moreover, assuming that $\tau = g(\hat{u}(\lambda))$, it also clearly holds that 
\begin{displaymath}
\lambda (g(\hat{u}(\lambda))-\tau)=0, \quad \lambda > 0, \quad g(\hat{u}(\lambda)) \leq \tau,
\end{displaymath}
which are the (necessary and sufficient) conditions to ensure that $\hat{u}(\lambda)$ is the solution of the constrained minimization problem (\ref{eq:constrainedproblem}), see e.g. \cite[Theorem 3.34]{ruszczynski2011nonlinear}.

Assume now that both $f$ and $g$ are differentiable. Using the normal equation $\nabla f(\hat u(\lambda))+\lambda \nabla g(\hat u(\lambda))=0$ it follows by the chain rule that:
\begin{displaymath}
\frac{\dd f(\hat u(\lambda))}{\dd g(\hat u(\lambda))} = \frac{\dd f/\dd\lambda }{\dd g/\dd \lambda} = \frac{\langle \nabla f(\hat u(\lambda)), \dd \hat u(\lambda)/\dd \lambda \rangle}{\langle\nabla g(\hat u(\lambda)), \dd \hat u(\lambda)/\dd \lambda\rangle } = -\lambda
\end{displaymath}
as announced.

\end{proof}

\begin{remark}
In the iterative algorithm (\ref{eq:method}) the sequence of parameters $\lambda_n$ is assumed to converge to the strictly positive value $\lambda$ present in problem (\ref{eq:problem}), i.e., $\lambda_n\to\lambda$. If, additionally, one imposes monotone convergence (which implies $\lambda_n>\lambda$), and one starts the iteration with a minimizer $u_0=\hat u(\lambda_0)$, one can surmise that a good approximation of the trade-off curve (slopes between $\lambda_0$ and $\lambda$) may be obtained. In this sense, the algorithm (\ref{eq:method}) follows a more interesting path to the solution of problem (\ref{eq:problem}) than a generic iterative algorithm (see Figure~\ref{fig1}, right panel). 
\end{remark}

\section{Convergence analysis}\label{sec:convergence}

The proof of convergence of algorithm (\ref{eq:method}) is similar to the proof of convergence of the classical proximal-gradient algorithm (algorithm (\ref{eq:method}) with $\lambda_n=\lambda$ constant) \cite{ComWa2005}. It is primarily based on the following three lemmas.

\begin{lemma} If $f:\R^d \rightarrow \R$ is convex with Lipschitz continuous gradient ($L$) then $\frac{1}{L}\nabla f$ is firmly non expansive:
\begin{equation}\label{fne}
\langle \frac{1}{L}\nabla f(u)-\frac{1}{L}\nabla f(v),u-v\rangle \geq \|\frac{1}{L}\nabla  f(u)-\frac{1}{L} \nabla f(v)\|_2^2
\quad\forall u,v\in\R^d
\end{equation}
\begin{proof}
See \cite[Part 2, Chapter X, Th. 4.2.2]{HiriartUrruty1993}.
\end{proof}
\end{lemma}

\begin{lemma}
    Let $h:\R^d\rightarrow \R\cup\{+\infty\}$ be a convex function. Then $u^+=\prox_{h}(u^-+\Delta)$ if and only if
    \begin{equation}\label{eq:prox_ine}
        \|u^+-u\|^2\leq \|u^--u\|^2-\|u^+-u^-\|^2+2\langle u^+-u,\Delta\rangle+ 2h(u)-2h(u^+), \quad \forall \ u\in\R^d.
    \end{equation}
\end{lemma}
\begin{proof}
The relation $u^+=\prox_{h}(u^-+\Delta)$ is equivalent to the inclusion $u^--u^++\Delta\in \partial h(u^+)$. Hence one has:
\begin{displaymath}
h(u)\geq h(u^+)+\langle u^--u^++\Delta,u-u^+\rangle.
\end{displaymath}
The inner product $\langle u^--u^+,u-u^+\rangle$ can be re-arranged as a combination of three squares.
\end{proof}
    
\begin{lemma}\label{lemma-inequality}
Let $\{a_n\}_{n\in\N}, \{\epsilon_n\}_{n\in\N}\subset \R_{\geq 0}$ with  $\sum_n a_n < \infty$.
If one furthermore has $\epsilon_{n+1}^2 \leq \epsilon_n^2 + 2 a_n \epsilon_{n+1}$ then $\{\epsilon_n\}_{n\in\N}$ is a bounded sequence and $\epsilon_{n+1}\leq \epsilon_n+2a_n$.
\end{lemma}    
\begin{proof}
The inequality can be rewritten as:$(\epsilon_{n+1}-a_n)^2 \leq \epsilon_n^2 + a_n^2$ which implies
\begin{displaymath}
    |\epsilon_{n+1}-a_n|\leq \sqrt{\epsilon_n^2 + a_n^2}\leq \epsilon_n + a_n.
\end{displaymath}
If $\epsilon_{n+1}-a_n\geq 0$ one finds $\epsilon_{n+1}\leq \epsilon_n+2a_n$.
If $\epsilon_{n+1}-a_n\leq 0$ one finds $\epsilon_{n+1}\leq a_n\leq \epsilon_n+2a_n$ also.
Finally, this implies $\epsilon_{n+1} \leq \epsilon_0 + \sum_{k=0}^n a_n \leq \epsilon_0 + \sum_{k=0}^\infty a_n < \infty$ independently of $n$.
\end{proof}
    
\begin{theorem}\label{thm:1} Under the assumptions of Problem~\ref{problem}, let $\{u_n\}_{n\in\N}$ be the sequence generated by algorithm \eqref{eq:method}. Assume that $\alpha\in(0,2/L)$ and that 
    \begin{equation}\label{eq:lambda}
        \overline{\lambda}=\sum_{n=0}^{\infty}|\lambda_n-\lambda|<\infty.
        \end{equation}
    Then the sequence $\{u_n\}_{n\in\N}$ converges to a solution $\hat{u}\in\R^d$ of problem \eqref{eq:problem}.
\end{theorem}

\begin{proof}
Let $n\in \N$ and $\hat{u}$ a minimizer of $F_\lambda$. We start by applying twice Lemma 1 considering first $(u^+, u^-, u, \Delta, h ) = (u_{n+1}, u_n, \hat{u} -\alpha \nabla f (u_n), \alpha \lambda_n g )$ and then $(u^+, u^-, u, \Delta, h ) = (\hat{u}, \hat{u}, u_{n+1}, -\alpha \nabla f (\hat{u}), \alpha \lambda g )$. This gives the two following relations 
\begin{align}
     \|u_{n+1}-\hat{u}\|^2&\leq \|u_n-\hat{u}\|^2-\|u_{n+1}-u_n\|^2 \nonumber\\
     &\qquad-2\alpha\langle u_{n+1}-\hat{u},\nabla f(u_n)\rangle+ 2\alpha \lambda_n \left(g(\hat{u})-g(u_{n+1})\right), \nonumber\\
      0 &\leq  2 \alpha \langle u_{n+1}-\hat{u},\nabla f (\hat{u})\rangle+ 2\alpha \lambda \left( g(u_{n+1})-g(\hat{u}) \right).\nonumber
\end{align}
Combining the first relation and $\lambda_n/\lambda$ times the second, one finds:
\begin{align}
    \|u_{n+1}-\hat{u}\|^2&\leq \|u_n-\hat{u}\|^2-\|u_{n+1}-u_{n}\|^2+\frac{2\alpha}{\lambda}\langle u_{n+1}-\hat{u},\lambda_n \nabla f(\hat{u})-\lambda\nabla f (u_{n})\rangle \nonumber\\
    & =\|u_n-\hat{u}\|^2-\|u_{n+1}-u_{n}\|^2+\frac{2\alpha}{\lambda}\langle u_{n+1}-\hat{u},\lambda_n \nabla f(\hat{u})-\lambda\nabla f (\hat{u})\rangle \nonumber\\
    &\qquad\qquad\qquad+\frac{2\alpha}{\lambda}\langle u_{n+1}-\hat{u},\lambda \nabla f(\hat{u})-\lambda\nabla f (u_{n})\rangle \nonumber\\
    &=\|u_n-\hat{u}\|^2-\|u_{n+1}-u_{n}\|^2+\frac{2\alpha}{\lambda}(\lambda_n-\lambda)\langle u_{n+1}-\hat{u},\nabla f(\hat{u})\rangle \nonumber \\
& \qquad\qquad    +2\alpha\langle u_{n+1}-\hat{u}, \nabla f(\hat{u})-\nabla f (u_{n})\rangle.\nonumber
\end{align}
The last inner product can be bounded above by
\begin{displaymath}
\begin{array}{lcl}
\displaystyle\langle\hat u- u_{n+1},\nabla f(u_n)-\nabla f(\hat u)\rangle
&=& \displaystyle\langle\hat u- u_{n},\nabla f(u_n)-\nabla f(\hat u)\rangle\\
&&\displaystyle+\langle u_n- u_{n+1},\nabla f(u_n)-\nabla f(\hat u)\rangle\\[3mm]
&\stackrel{(\ref{fne})}{\leq}&\displaystyle\frac{-1}{L}\|\nabla f(u_n)-\nabla f(\hat u)\|_2^2\\
&&\displaystyle\qquad+\langle u_n- u_{n+1},\nabla f(u_n)-\nabla f(\hat u)\rangle\\[3mm]
&=&\displaystyle\langle \sqrt{L}(u_n- u_{n+1})-\frac{1}{\sqrt{L}}(\nabla f(u_n)-\nabla f(\hat u)),\\
&&\displaystyle\qquad\qquad \frac{1}{\sqrt{L}}(\nabla f(u_n)-\nabla f(\hat u))\rangle\\[3mm]
{\footnotesize\langle a,b\rangle=\frac{\|a+b\|_2^2-\|a-b\|_2^2}{4}}&=&\displaystyle \frac{L}{4}\|u_n- u_{n+1}+0\|_2^2-\frac{1}{4}\|u_n- u_{n+1}-\frac{2}{\sqrt{L}}\ldots\|_2^2\\[3mm]
&\leq&\displaystyle \frac{L}{4}\|u_n- u_{n+1}\|_2^2.
\end{array}
\end{displaymath}
Hence one finds:
\begin{align}
  \|u_{n+1}-\hat{u}\|^2&\leq \|u_n-\hat{u}\|^2-\|u_{n+1}-u_{n}\|^2+\frac{2\alpha}{\lambda}(\lambda_n-\lambda)\langle u_{n+1}-\hat{u},\nabla f(\hat{u})\rangle \nonumber\\
& \qquad\qquad    +2\alpha \frac{L}{4}\|u_n-u_{n+1}\|^2 \nonumber\\
    &=\|u_n-\hat{u}\|^2-(1-\frac{\alpha}{2L})\|u_{n+1}-u_{n}\|^2 \nonumber\\
    & \qquad +\frac{2\alpha}{\lambda}(\lambda_n-\lambda)\langle u_{n+1}-\hat{u},\nabla f(\hat{u})\rangle.\nonumber
\end{align}
Using Cauchy-Schwartz on the last scalar product, we finally have
\begin{equation}\label{fundamentalinequality}
    \|u_{n+1}-\hat{u}\|^2\leq 
    \|u_n-\hat{u}\|^2-(1-\frac{\alpha}{2L})\|u_{n+1}-u_{n}\|^2+2C|\lambda_n-\lambda|\times \|u_{n+1}-\hat{u}\|
\end{equation}
with $C$ independent of $n$. Lemma~\ref{lemma-inequality} implies that the sequence $(u_n)_{n}$ is bounded when $0<\alpha<2/L$. Hence there is a converging subsequence $u_{n_j}\stackrel{j\to\infty}{\rightarrow}u^\dagger$.

Using the boundedness of the sequence $(u_n)_n$, relation (\ref{fundamentalinequality}) implies
\begin{displaymath}
    (1-\frac{\alpha}{2L})\|u_{n+1}-u_{n}\|^2\leq 
    \|u_n-\hat{u}\|^2-\|u_{n+1}-\hat{u}\|^2+2\tilde C|\lambda_n-\lambda|
\end{displaymath}
and
\begin{displaymath}
    (1-\frac{\alpha}{2L})\sum_{n=0}^N\|u_{n+1}-u_{n}\|^2\leq 
    \sum_{n=0}^N\|u_n-\hat{u}\|^2-\|u_{n+1}-\hat{u}\|^2+2\tilde C|\lambda_n-\lambda|\leq C_2
\end{displaymath}
independently of $N$. Hence $\|u_{n+1}-u_n\|\to 0$ for $n\to\infty$.

Thus $u_{n_j+1}$ also tends to $u^\dagger$ and the continuity of the right hand side of (\ref{eq:method}) implies that $u^\dagger$ satisfies the fixed point relation (is a minimizer).
The fundamental inequality (\ref{fundamentalinequality}) is valid for $\hat u=u^\dagger$ and it implies (using the lemma~\ref{lemma-inequality})
\begin{displaymath}
    \|u_{n+1}-u^\dagger\|\leq \|u_n- u^\dagger\|+2 a_n
\end{displaymath}
(with $a_n=C |\lambda_n-\lambda |$) and
\begin{displaymath}
    \|u_{N}-u^\dagger\|\leq \|u_M- u^\dagger\|+2\sum_{k=M}^N a_n
\end{displaymath}
which implies convergence of the whole sequence.
\end{proof}

\begin{theorem}\label{thm:2}
Let $\{u_n\}_{n\in\N}$ be the sequence generated by \eqref{eq:method} with $\alpha \in (0,1/L)$ and where condition \eqref{eq:lambda} holds. Let $\hat{u}$ be a solution to problem \eqref{eq:problem}. Then the following convergence rate on the cost function value is obtained
\begin{equation}\label{eq:rate}
    F_\lambda\left(\frac{1}{n+1}\sum_{i=0}^n u_{i+1}\right)-F_\lambda(\hat{u})\leq \frac{\|u_0-\hat{u}\|^2 + M\overline{\lambda}}{2\alpha(n+1)}, \quad n=0,1,\ldots
\end{equation}
with $M  = \sup_{i\in \N}|g(u_i) - g(\hat{u})|$. 
\end{theorem}

\begin{proof}
By applying Lemma \eqref{eq:prox_ine} with $u^+ = u_{n+1}$, $u^{-} = u_n$, $u=\hat u$, $\Delta = -\alpha \nabla f(u_n)$, $h = \alpha \lambda_n  g$ for a given $n\in \N$ leads to
\begin{align}
        \|u_{n+1}-\hat{u}\|^2 &\leq \|u_n-\hat{u}\|^2-\|u_{n+1}-u_n\|^2 \nonumber\\
        &\qquad +2\alpha\left [\langle \hat{u}-u_{n+1},\nabla f (u_n)\rangle + \lambda_n g(\hat{u})-\lambda_n g(u_{n+1}) \right] . \label{eq:thm2_1}
\end{align}
We now need to give a majoration of the terms between brackets we denote by $\Gamma_n$. To do so, we use the classical convexity inequality and descent lemma both applied to $f$. 
\begin{align}\label{eq:thm2_2}
    \Gamma_n
    &= \langle \hat{u}- u_n,\nabla f (u_n)\rangle + \langle u_n - u_{n+1},\nabla f (u_n)\rangle + \lambda_n g(\hat{u})- \lambda_n g(u_{n+1})\nonumber\\
    &\leq f(\hat{u}) - f(u_n) + \langle u_n-u_{n+1},\nabla f (u_n)\rangle + \lambda_n g(\hat{u})-\lambda_n g(u_{n+1}) \nonumber\\
    & = f(\hat{u}) + \lambda_n g(\hat{u}) + \frac{L}{2}\|u_{n+1} - u_n\|^2  \nonumber\\
    &\qquad - \left[f(u_n) + \langle u_{n+1}-u_n,\nabla f (u_n) \rangle + \frac{L}{2}\|u_{n+1} - u_n\|^2 \right] - \lambda_n g(u_{n+1}) \nonumber\\
    & \leq f(\hat{u}) + \lambda_n g(\hat{u}) + \frac{L}{2}\|u_{n+1} - u_n\|^2  -f(u_{n+1}) - \lambda_n g(u_{n+1}) \nonumber\\
    & = F_\lambda(\hat{u}) - F_\lambda(u_{n+1}) + \frac{L}{2}\|u_{n+1} - u_n\|^2 + (\lambda- \lambda_n) \left(g(u_{n+1})-g(\hat{u})\right). 
\end{align}
Replacing \eqref{eq:thm2_2} in \eqref{eq:thm2_1} and using the fact that $1-\alpha L \in [0,1)$ gives 
\begin{align}\label{eq:thm2_2b}
    \|u_{n+1}-\hat{u}\|^2&\leq \|u_n-\hat{u}\|^2-\|u_{n+1}-u_n\|^2 \nonumber\\
    & 
    + 2\alpha \left[ F_\lambda(\hat{u}) - F_\lambda(u_{n+1}) + \frac{L}{2}\|u_{n+1} - u_n\|^2 + (\lambda- \lambda_n) \left(g(u_{n+1})-g(\hat{u})\right) \right] \nonumber\\
    & = \|u_n-\hat{u}\|^2-(1 - \alpha L) \|u_{n+1} - u_n\|^2 + 2\alpha \left( F_\lambda(\hat{u}) - F_\lambda(u_{n+1}) \right)\nonumber\\
    & \qquad + 2\alpha  (\lambda- \lambda_n) \left(g(u_{n+1})-g(\hat{u})\right) \nonumber\\
    &\leq \|u_n-\hat{u}\|^2 + 2\alpha \left( F_\lambda(\hat{u}) - F_\lambda(u_{n+1}) \right) + 2\alpha  (\lambda- \lambda_n) \left(g(u_{n+1})-g(\hat{u})\right)\nonumber
\end{align}
which leads to
\begin{equation} \label{eq:thm2_3}
    2\alpha \left( F_\lambda(u_{n+1}) - F_\lambda(\hat{u}) \right) \leq \|u_n-\hat{u}\|^2 - \|u_{n+1}-\hat{u}\|^2 + 2\alpha  (\lambda- \lambda_n) \left(g(u_{n+1})-g(\hat{u})\right). 
\end{equation}
Summing relation \eqref{eq:thm2_3} from 0 to an arbitrary $n\in \N$ yields:
\begin{align}\label{eq:thm2_4}
    2 \alpha \sum_{i=0}^n F_\lambda(u_{i+1}) - 2\alpha (n+1) F_\lambda(\hat{u})  &\leq \|u_0 - \hat{u} \|^2 - \|u_{n+1} - \hat{u} \|^2 \nonumber\\
    &\qquad +  \sum_{i=0}^n (\lambda- \lambda_i)\left(g(u_{i+1})-g(\hat{u})\right) \nonumber\\ 
    &\leq \|u_0 - \hat{u} \|^2 +  M \overline{\lambda}. 
\end{align}
The convexity of $F_\lambda$ (as a positive linear combination of $f,g$ convex) enables to invoke the Jensen's inequality so as to lower-bound the left term of \eqref{eq:thm2_4}. We thus deduce that
\begin{equation}
    2\alpha (n+1) F_\lambda\left(\frac{1}{n+1}\sum_{i=0}^n u_{i+1}\right) - 2\alpha (n+1) F_\lambda(\hat{u}) \leq  \|u_0 - \hat{u} \|^2 +  M \overline{\lambda} 
\end{equation}
and $\eqref{eq:rate}$ is obtained by simply making the division by $2\alpha(n+1)$. 
\end{proof}

\begin{remark}
Under the additional assumption that $g$ is continuous on the entire domain, method \eqref{eq:method} can be interpreted as an inexact forward--backward algorithm applied to problem \eqref{eq:problem}, by employing the concept of {\it $\epsilon_n-$approximation of $1-$type} \cite{Salzo2012,Villa2013}. For any given $x\in\R^d$, let us introduce the function
\[
\varphi_\lambda^{(x)}(z) = \frac{1}{2} \| z-x\|^2 + \lambda g(z).
\]
Let $y\in\R^d$ be the (exact) proximal point of $\lambda g$ evaluated at $x$, which is defined by minimizing $\varphi_\lambda^{(x)}$:
\[
 y = \prox_{\lambda g}(x) \quad \Leftrightarrow \quad y = \argmin_{z} \varphi_\lambda^{(x)}(z)  \quad \Leftrightarrow \quad 0 \in \partial \varphi_\lambda^{(x)}(y) .
\]
Then, given $\epsilon>0$, an $\epsilon-$approximation of $1-$type of $y$ is any point $\tilde{y}\in\R^d$ such that
\[
\tilde{y} \approx_{1}^{\epsilon} y \quad \Leftrightarrow \quad 0 \in \partial_\varepsilon \varphi_\lambda^{(x)}(\tilde{y}),
\]
where the $\epsilon-$subdifferential is defined as
\[
\partial_\epsilon F(z) = \{ \xi\in\R^d: F(x)\geq F(z) + \langle \xi,x-z\rangle -\epsilon, \ \forall x \in \R^n \}.
\]
For each $n$, we define
\begin{align*}
    \varphi^{(n)}_\lambda(u)&=\frac{1}{2}\|u-(u_n-\alpha\nabla f(u_n))\|^2+\lambda\alpha g(u)\\
    \varphi_{\lambda_n}^{(n)}(u)&=\frac{1}{2}\|u-(u_n-\alpha\nabla f(u_n))\|^2+\lambda_n\alpha g(u)\\
        u_{\lambda}^{(n)}&= \argmin_u \ \varphi_\lambda^{(n)}(u)=\prox_{\lambda\alpha g}(u_n-\alpha\nabla f(u_n)).
\end{align*}
We note that $\varphi^{(n)}_\lambda$ is the function to be minimized at each step of the forward--backward algorithm applied to problem \eqref{eq:problem}, $u_{\lambda}^{(n)}$ is the exact proximal-gradient point obtained by minimizing $\varphi^{(n)}_\lambda$, and $\varphi_{\lambda_n}^{(n)}$ is the function that our proposed method \eqref{eq:method} minimizes in place of $\varphi^{(n)}_\lambda$. Then, we can write down the following implications:
\begin{align*}
    u_{n+1}&=\prox_{\alpha\lambda_n g}(u_n-\alpha \nabla f(u_n)) \\
    &\Leftrightarrow \quad u_{n+1} = \argmin_u \ \varphi_{\lambda_n}^{(n)}(u)\\
    &\Leftrightarrow \quad 0\in\partial \varphi_{\lambda_n}^{(n)}(u_{n+1}) \\
    &\Leftrightarrow \quad \varphi_{\lambda_n}^{(n)}(u)\geq \varphi_{\lambda_n}^{(n)}(u_{n+1}), \quad \forall \ u\in\R^d\\
    &\Rightarrow \quad \varphi_{\lambda_n}^{(n)}(u_{\lambda}^{(n)})\geq \varphi_{\lambda_n}^{(n)}(u_{n+1})\\
    &\Rightarrow \quad \varphi_{\lambda}^{(n)}(u_{\lambda}^{(n)})\geq \varphi_{\lambda}^{(n)}(u_{n+1})+\alpha(\lambda-\lambda_n)(g(u_{\lambda}^{(n)})-g(u_{n+1}))\\
    &\Rightarrow \quad \varphi_{\lambda}^{(n)}(u)\geq \varphi_{\lambda}^{(n)}(u_{n+1})-\alpha|\lambda-\lambda_n|\times |g(u_{\lambda}^{(n)})-g(u_{n+1})|, \quad \forall \ u\in\R^d,
\end{align*}
where the last inequality follows from the fact that $u_{\lambda}^{(n)}$ is the unique minimizer of $\varphi_{\lambda}^{(n)}$. From the continuity of the operator $T(\alpha,\lambda,u)=\prox_{\alpha \lambda g}(u-\alpha \nabla f(u))$ with respect to $\lambda,u$, the boundedness of $\{u_n\}_{n\in\N}$, and the fact that $\lambda_n\rightarrow \lambda$, it follows that the sequence $\{u_{\lambda}^{(n)}\}_{n\in\N}$ is bounded. Since $g$ is continuous by assumption, we conclude that $|g(u_{\lambda}^{(n)})-g(u_{n+1})|$ is also bounded. Denoting with $M=\sup_n |g(u_{\lambda}^{(n)})-g(u_{n+1})|$, we have
\begin{align}
    u_{n+1}&=\prox_{\alpha\lambda_n g}(u_n-\alpha \nabla f(u_n))\nonumber \\
    &\Rightarrow \quad \varphi_{\lambda}^{(n)}(u)\geq \varphi_{\lambda}^{(n)}(u_{n+1})-\alpha M|\lambda-\lambda_n|, \quad \forall \ u\in\R^d,\nonumber\\
    &\Rightarrow \quad 0\in\partial_{\epsilon_n} \varphi_{\lambda}^{(n)}(u_{n+1}), \quad \text{where }\epsilon_n=\alpha M|\lambda-\lambda_n|\nonumber\\
    &\Rightarrow \quad u_{n+1}\approx_1^{\epsilon_n} \prox_{\alpha \lambda g}(u_n-\alpha \nabla f(u_n)).\label{eq:equivalent}
\end{align}
In the above form \eqref{eq:equivalent}, method \eqref{eq:method} can be interpreted as a special instance of the inexact proximal-gradient method proposed in \cite[Equation 4]{Schmidt2011a}. In this light, the convergence rate provided in our Theorem \ref{thm:2} is coherent with the more general result \cite[Proposition 1]{Schmidt2011a}, which is also given in terms of the function value attained by the average of the iterates, although the constant multiplying the term $1/(n+1)$ is different from the one in our result. On the other hand, the convergence of the iterates is not given for the general method in \cite{Schmidt2011a}, whereas here we are able to guarantee convergence for the specific method \eqref{eq:method}.
\end{remark}

\section{Numerical experiments}\label{sec:experiments}

In order to support the theoretical arguments of section~\ref{sec:tradeoff}, we perform some numerical experiments demonstrating the described behaviour of the trade-off curve and of the iterates of algorithm (\ref{eq:method}). Our test problem is a simple deconvolution and denoising problem. A $128 \times 128$ greyscale image is degraded by convolving it with a $5 \times 5$ convolution kernel.
Furthermore the blurred image is corrupted with standard Gaussian noise scaled by $0.03$.

If one assumes that the original image has a sparse wavelet decomposition, one may try to recover the original image by solving the following $\ell_1$-norm penalized least squares optimization problem:
\begin{align}\label{eq:otimizationproblem}
    \hat{u} \in \argmin_{u} \lVert A W^* u - x_0 \rVert_2^2 + \lambda \lVert u \rVert_1,
\end{align}
where $A$ is the known blur matrix, $W$ is a 2D orthogonal wavelet transform (using Daubechies 3 wavelets \cite{Daubechies1992}), $W^*$ the corresponding inverse transform and $x_0$ the degraded image. The solution $\hat{u}$ of the problem~\eqref{eq:otimizationproblem} is then the restored image in the wavelet domain, which gives us the restored image $\hat{x} = W^*\hat{u}$.

We use algorithm~\eqref{eq:method} to reconstruct the original image using different sequences $\{\lambda_n\}_{n \in \N}$, which lead to a variant of the famous iterative soft-thresholding algorithm (ISTA) \cite{daubechies2004}. Indeed by setting
\begin{displaymath}
    f(u) = \lVert A W^* u - x_0 \rVert_2^2, 
    \qquad g(u) = \lVert u \rVert_1,
\end{displaymath}
we find that 
\begin{displaymath}
    \nabla f(u)=2WA^\ast(AW^\ast u-x_0)\quad\text{and}\quad (\prox_{\alpha\lambda g}(u))_i=
    \left\{
    \begin{array}{lcl}
       u_i+\alpha\lambda  &\quad&  u_i\leq -\alpha\lambda \\
       0  &&  |u_i|\leq \alpha\lambda\\
       u_i-\alpha\lambda  &&  u_i\geq \alpha\lambda \\
    \end{array}.
    \right.
\end{displaymath}
The starting point of the algorithm is the noisy blurry image in the wavelet domain, $u_0 = W x_0$.
The step size is chosen as $\alpha = 1/L$, where the Lipschitz constant $L$ of the function $f$ is  $L = 2 \lVert A W^* \rVert_2^2 = 2 \lVert A \rVert_2^2$ (since the wavelets form an orthogonal basis).

\subsection{Approximating the trade-off curve}

In order to find an adequate regularization parameter $\lambda$ for the classical proximal-gradient algorithm~\eqref{eq:proxgrad}, it is necessary to run the algorithm for several values of the parameter, plot the trade-off curve and choose a value that provides a good balance between the data mismatch and regularization. As already mentioned this is known as the L-curve method in the literature \cite{lawson1995, Hansen2001}. However, running the optimization algorithm (\ref{eq:proxgrad}) several times might be very time consuming.
Therefore a speed-up may be obtained if the trade-off curve can be generated by just running the algorithm once with a fitting sequence for the regularization parameter ($\lambda_n$ instead of a fixed value).

To demonstrate this behaviour, we first generate the trade-off curve by running the algorithm (\ref{eq:proxgrad}) multiple times with a fixed regularization parameter chosen from the interval $[10^{-3}, 10^{-1}]$ to have a reference for comparison. Then we choose three different sequences $(\lambda_n)_{n\in\N}$ with the same starting point converging to the ``optimal" parameter value $\lambda$ that we determined using the L-curve method. Our test sequences are
\begin{align}
\label{eq:lambdaSeq1}
    \lambda^1_n &= \lambda (1 + \frac{\beta}{n^\theta}) \quad \text{with } \theta = 1.01, \beta = 9 \\
\label{eq:lambdaSeq2}
    \lambda^2_n &= \max (\lambda, \mu \beta^n) \quad \text{with } \beta = 0.99, \mu = 10 \lambda \\
\label{eq:lambdaSeq3}
    \lambda^3_n &= \lambda (1 + \mu \beta^n) \quad \text{with } \beta = 0.9, \mu = 9.
\end{align}
As can be seen in Figure~\ref{fig:CurveNumeric} these sequences already cover parts of the trade-off curve as opposed to choosing a constant parameter. Since all sequences converge to $\lambda$ it is not surprising that they all stop at the same point of the curve.

\begin{figure}[h]
\centering\includegraphics[width=\textwidth]{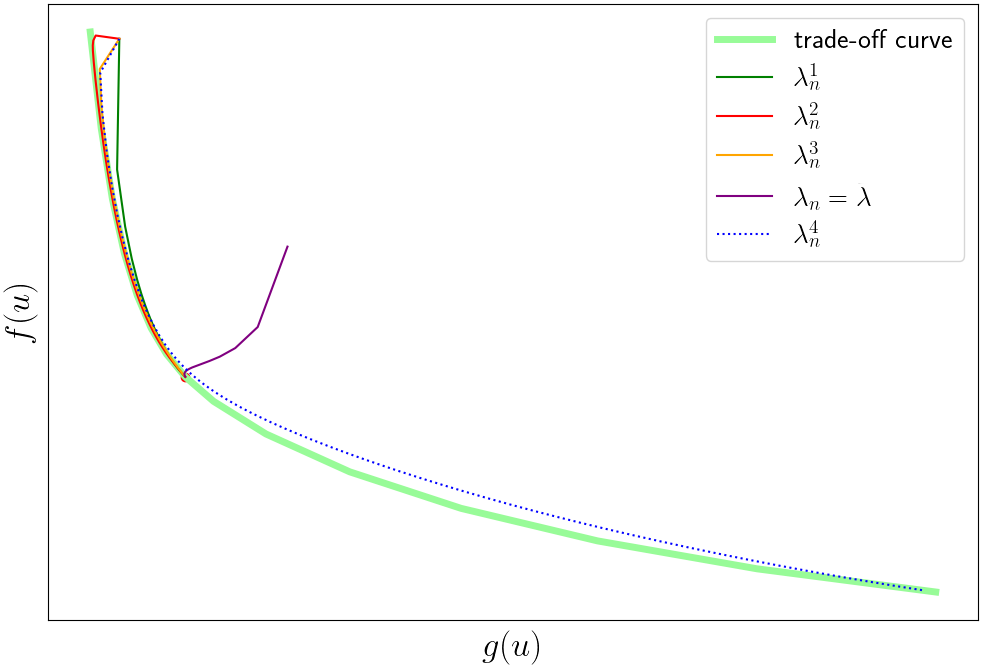}
\caption{Paths of the different sequences in the $g-f$ plane compared to the trade-off curve}\label{fig:CurveNumeric}
\end{figure}

However, in practice these sequences are not really applicable since the optimal value needs to be known beforehand. Therefore we tried a fourth sequence similar to the third one~\eqref{eq:lambdaSeq3} with $\lambda^4_0 = 10^{-1}$ defined by
\begin{align}
    \lambda^4_n = 10^{-3} (1 + 99 \cdot 0.9^n)
\end{align}
that converges to $10^{-3}$ covering a wider range of possible regularization parameters.
This sequence's path does not follow the trade-off curve perfectly but it is quite close (see Figure~\ref{fig:CurveNumeric}), such that by only executing the algorithm once an approximation of the trade-off curve can be generated.

\section{Conclusions}\label{sec:conclusions}

A proof of convergence of an iterative optimization algorithm for the composite problem (\ref{eq:problem}) was given. A special case of the algorithm of interest has already been proposed \cite{Hale2008} but no proof of convergence was given. In addition, we derived a convergence rate estimate. We also highlight the relation with the so-called inexact proximal-gradient methods, in particular with algorithms based on the notion of $\epsilon$-subdifferential.

The advantage of the proposed method (with varying $\lambda_n$) is not that it necessarily converges faster than the usual proximal gradient algorithm, but that it traces out a more interesting path in the penalty-misfit plane. In this way, an approximation of the trade-off curve can be made at the cost of computing just a single minimizer, and the intermediate iterates $u_n$ are of some use for balancing the data mismatch and regularization terms.

Of course searching through a very wide range of possible values for the parameter may not be reliable, but at least in our test case covering few different orders of magnitude (e.g. $\lambda_n \in [10^{-1}, 10^{-3}]$) was not a problem. The particular choice of the sequence did not seem to have noticeable effect on the behaviour.

\begin{acknowledgement}
The authors would like to thank the organizers and participants of the workshop on Advanced Techniques in Optimization for Machine learning and Imaging (ATOMI, Rome, 20-24 June, 2022) during which the present work was initiated. This work was supported by the Fonds de la Recherche Scientifique - FNRS under Grant CDR J.0122.21. LR was supported by the Air Force Office of Scientific Research under award number FA8655-20-1-7027, and acknowledges the support of
Fondazione Compagnia di San Paolo.  
\end{acknowledgement}

\bibliography{main.bib}{}
\bibliographystyle{spmpsci}

\end{document}